\documentclass[12pt]{article}
\pdfoutput=1

\usepackage[utf8]{inputenc}
\usepackage[T1]{fontenc}
\usepackage{mathpazo}

\usepackage[english]{babel}

\usepackage[a4paper,margin=2.5cm]{geometry}
\usepackage{parskip}

\newcommand{\affiliation}{\footnote}

\usepackage[shortlabels]{enumitem}
\setlist{itemsep=0ex,topsep=0ex,parsep=0.4ex}

\usepackage{amsmath,amsfonts,amssymb,amsthm,mathtools,thmtools,thm-restate,bbm,centernot}

\usepackage{xcolor}
\definecolor{blue}{RGB}{0,70,140}
\definecolor{green}{RGB}{100,140,0}
\definecolor{red}{RGB}{190,10,50}

\colorlet{structurecolor}{red}
\colorlet{linkcolor}{blue}
\colorlet{citecolor}{green}

\usepackage[hyphens]{url} 
\usepackage[hidelinks,colorlinks,pagebackref]{hyperref} 
\hypersetup{citecolor=citecolor,linkcolor=linkcolor,urlcolor=linkcolor}
\usepackage[capitalise,nameinlink,noabbrev,compress]{cleveref} 

\renewcommand*{\backref}[1]{}
\renewcommand*{\backrefalt}[4]{
	\ifcase #1 Not cited.%
	\or $\uparrow$#2%
	\else $\uparrow$#2%
	\fi%
}

\declaretheorem[name=Definition,numberwithin=section,style=definition]{definition}
\declaretheorem[name=Theorem,numberlike=definition,style=plain]{theorem}
\declaretheorem[name=Proposition,numberlike=theorem,style=plain]{proposition}
\declaretheorem[name=Lemma,numberlike=theorem,style=plain]{lemma}

\declaretheorem[name=Conjecture,numberlike=theorem,style=plain]{conjecture}

\newcommand{\defn}[1]{\textcolor{structurecolor}{\emph{#1}}}

\mathtoolsset{centercolon}

\let\emptyset\varnothing

\renewcommand{\l}{\mathopen{}\mathclose\bgroup\left}
\renewcommand{\r}{\aftergroup\egroup\right}

\newcommand{\st}{\ifnum\currentgrouptype=16 \mathrel{}\middle|\mathrel{}\else\mathrel{|}\fi}

\DeclarePairedDelimiter{\set}{\{}{\}}
\DeclarePairedDelimiter{\abs}{\lvert}{\rvert}
\DeclarePairedDelimiter{\floor}{\lfloor}{\rfloor}
\DeclarePairedDelimiter{\ceil}{\lceil}{\rceil}
\DeclarePairedDelimiter{\inprod}{\langle}{\rangle}

\newcommand{\subs}{\subseteq}

\newcommand{\sm}{\setminus}
\newcommand{\eps}{\varepsilon}

\makeatletter
\newcommand{\generate}[4]{%
  \def\@tempa{#1} 
  \count@=`#3
  \loop
  \begingroup\lccode`?=\count@
  \lowercase{\endgroup\@namedef{\@tempa ?}{#2{?}}}%
  \ifnum\count@<`#4
  \advance\count@\@ne
  \repeat
}
\makeatother

\generate{b}{\mathbb}{A}{Z}
\generate{c}{\mathcal}{A}{Z}

\newcommand{\DeclareMath}[2]{\newcommand{#1}{\mathnormal{#2}}}

\DeclareMath{\N}{\mathbb{N}}
\DeclareMath{\Z}{\mathbb{Z}}
\DeclareMath{\F}{\mathbb{F}}

\DeclareMathOperator{\rank}{rank}

\title{Circuit decompositions of binary matroids}
\author{Bryce Frederickson\affiliation{Department of Mathematics, Emory University, Atlanta, Georgia, 30322 (\textsf{\href{mailto:bfrede4@emory.edu}{bfrede4@emory.edu}}).} \and Lukas Michel\affiliation{Mathematical Institute, University of Oxford, United Kingdom (\textsf{\href{mailto:michel@maths.ox.ac.uk}{michel@maths.ox.ac.uk}}).}}
\date{24 June 2023}

\begin{document}
\maketitle

\begin{abstract}
    Given a simple Eulerian binary matroid $M$, what is the minimum number of disjoint circuits necessary to decompose $M$? We prove that $\abs{M} / (\rank(M) + 1)$ many circuits suffice if $M = \mathbb F_2^n \setminus \{0\}$ is the complete binary matroid, for certain values of $n$, and that $\cO(2^{\rank(M)} / (\rank(M) + 1))$ many circuits suffice for general $M$. We also determine the asymptotic behaviour of the minimum number of circuits in an odd-cover of $M$.
\end{abstract}

\section{Introduction}

Erd\H{o}s and Gallai conjectured that the edge set of any graph on $n$ vertices can be decomposed into $\cO(n)$ edge-disjoint cycles and edges \cite{erdHos1983some}. Equivalently, this says that any Eulerian graph can be decomposed into $\cO(n)$ edge-disjoint cycles. Despite receiving a lot of attention, the Erd\H{o}s-Gallai Conjecture remains a major open problem in the area of graph decompositions. While a straightforward greedy argument that iteratively removes largest cycles yields a decomposition of size $\cO(n \log n)$, it was only in 2014 that Conlon, Fox, and Sudakov \cite{conlon2014cycle} improved this upper bound to $\cO(n \log \log n)$. More recently, Buci\'c and Montgomery \cite{bucic2022towards} showed that $\cO(n \log^\star n)$ cycles suffice, where $\log^\star n$ is the iterated logarithm function.

Due to the difficulty of this problem, many variations of it have been considered. For example, if cycles can share edges, Fan proved that $\floor{(n-1)/2}$ cycles suffice to cover the edges of any Eulerian graph \cite{fan2003covers}. In fact, the cover can be chosen so that every edge is covered an odd number of times. In a similar vein, Pyber proved that any graph can be covered by $n-1$ cycles and edges \cite{pyber1985erdHos}.

In this note, we consider a matroid analogue of the cycle decomposition question: what is the minimum number of disjoint circuits necessary to decompose a matroid? We focus on (simple) matroids representable over the finite field $\F_2$. Up to isomorphism, such matroids are equivalent to \defn{(simple) binary matroids}, which are subsets $M \subs \F_2^n \sm \set{0}$ for some $n \ge 1$. In this setting, $M$ is \defn{Eulerian} if $\sum_{x \in M} x = 0$, and a subset $N \subs M$ is a \defn{circuit} if $N$ is a minimal non-empty Eulerian subset of $M$ with respect to inclusion.

We want to construct a circuit decomposition of $M$, that is, a small collection of disjoint circuits whose union is $M$. Observe that $M$ admits a circuit decomposition if and only if $M$ is Eulerian. In this case, we denote by $c(M)$ the minimum number of circuits in such a decomposition.

To obtain a lower bound on $c(M)$, note that every proper subset of a circuit in $M$ is linearly independent. For any binary matroid $M$, the \defn{rank} of $M$, denoted by $\rank(M)$, is the size of a largest linearly independent subset of $M$. Thus, any circuit in $M$ can have size at most $\rank(M) + 1$, which implies that $c(M) \ge \abs{M} / (\rank(M) + 1)$. For an Eulerian binary matroid $M$ of size $\Theta(2^{\rank(M)})$, this lower bound gives $c(M) \ge \Theta(2^{\rank(M)} / (\rank(M) + 1))$. We prove a matching upper bound.

\begin{theorem}\label{thm:decompqrdivr}
    For every Eulerian binary matroid $M \subs \F_2^n \sm \set{0}$ it holds that
    \begin{equation*}
        c(M) = \cO\l(\frac{2^{\rank(M)}}{\rank(M) + 1}\r).
    \end{equation*}
\end{theorem}

In fact, for certain values of $n$, we show that $M = \mathbb F_2^n \sm \set{0}$, which we call the \defn{complete binary matroid} of dimension $n$,\footnote{This is equivalent to the projective geometry $PG(n-1,2)$ in the literature \cite{oxley2006matroid}.} can be decomposed into exactly $\abs{M} / (\rank(M) + 1)$ many circuits, where $\abs{M} = 2^n - 1$ and $\rank(M) = n$.

\begin{theorem}\label{thm:decompcomplete}
    Let $p$ be an odd prime for which the multiplicative order of $2$ modulo $p$ is $p-1$, and let $M \subs \F_2^{p-1} \sm \set{0}$ be the complete binary matroid of dimension $p-1$. Then
    \begin{equation*}
        c(M) = \frac{2^{p-1}-1}{p}.
    \end{equation*}
\end{theorem}

For arbitrary Eulerian binary matroids, we prove the following upper bound on the size of a circuit decomposition.

\begin{theorem}\label{thm:decompmlogrdivroverq}
    For every Eulerian binary matroid $M \subs \F_2^n \sm \set{0}$ it holds that
    \begin{equation*}
        c(M) \le (1 + o(1)) \frac{\abs{M} \log(\rank(M))}{\log \abs{M}} \quad \text{ as } \quad |M| \to \infty.
    \end{equation*}
\end{theorem}

This bound is the correct order of magnitude for certain sparse binary matroids. For instance, if $M$ consists of $k$ independent copies of $\F_2^2 \sm \set{0}$, then $c(M) = k$ and
\begin{equation*}
    \frac{\abs{M} \log(\rank(M))}{\log \abs{M}} = \frac{3k \log(2k)}{\log(3k)} = (3 + o(1)) k.
\end{equation*}

In addition to circuit decompositions, we will also consider circuit odd-covers. For a graph $G$, an odd-cover is a collection of graphs on the same vertex set that covers each edge of $G$ an odd number of times and each non-edge of $G$ an even number of times. As mentioned above, every $n$-vertex Eulerian graph has an odd-cover with $\floor{(n-1)/2}$ cycles. More recently, Borgwardt, Buchanan, Culver, Frederickson, Rombach, and Yoo \cite{BBCFRY23} proved that every Eulerian graph of maximum degree $\Delta$ has an odd-cover with $\Delta$ cycles. Odd-covers were introduced by Babai and Frankl \cite{babai1988linear} and were also studied in \cite{buchanan2021subgraph} and \cite{buchanan2022odd}.

A \defn{circuit odd-cover} of a binary matroid $M$ is a collection of circuits $C_1, \dots, C_t \subs \F_2^n$ such that $C_1 \oplus \dots \oplus C_t = M$ where $A \oplus B$ denotes the symmetric difference of $A$ and $B$. In such an odd-cover, the elements of $M$ are covered an odd number of times while the elements of $\F_2^n \sm M$ are covered an even number of times. Note that, similar to the decomposition setting, the condition that $M$ is Eulerian is necessary and sufficient for the existence of a circuit odd-cover of $M$.

We denote by $c_2(M)$ the minimum number of circuits in a circuit odd-cover of $M$. Since every circuit decomposition is also a circuit odd-cover, we have $c_2(M) \le c(M)$. We can obtain the following natural lower bound for $c_2(M)$.

\begin{proposition}\label{prop:densitylowerboundoverq}
    For every Eulerian binary matroid $M \subs \F_2^n \sm \set{0}$ it holds that
    \begin{equation*}
        c_2(M) \ge \max_{N \subs M} \ceil*{\frac{\abs{N}}{\rank(N) + 1}}.
    \end{equation*}
\end{proposition}

\begin{proof}
    Consider a circuit odd-cover $C_1, \dots, C_t$ of $M$. For every subset $N \subs M$, each $C_i$ intersects $N$ in at most $\rank(N)+1$ elements since every proper subset of $C_i$ is linearly independent. The elements of $N$ must each be covered by $C_1, \ldots, C_t$ an odd number of times, so in particular, they must each be covered at least once. This implies that $t \cdot (\rank(N)+1) \ge \abs{N}$.
\end{proof}

The lower bound given in \cref{prop:densitylowerboundoverq} is closely related to the \defn{arboricity} of $M$, denoted $a(M)$, which is the minimum $t$ such that $M$ can be expressed as the union (or equivalently, as the symmetric difference) of $t$ linearly independent sets. In the case of graphic matroids, a decomposition of the matroid into independent sets coincides with a decomposition of the edge set of a corresponding graph into forests, whence the name arboricity. A celebrated theorem of Edmonds \cite{edmonds1965minimum} asserts that
\begin{equation*}
    a(M) = \max_{\emptyset \neq N \subs M} \ceil*{\frac{\abs{N}}{\rank(N)}}.
\end{equation*}
Since $\abs{N} \le 2^{\rank(N)}$, we have by \cref{prop:densitylowerboundoverq} that
\begin{equation}\label{eq:arblowerboundoddcover}
    c(M) \ge c_2(M) \ge (1 + o(1)) a(M) \quad \text{ as } \quad a(M) \to \infty.
\end{equation}
For $c(M)$, we cannot hope to attain this lower bound. For instance, if $M \subs \F_2^n \sm \set{0}$ consists of $k$ independent copies of $\F_2^s \sm \set{0}$, then $c(M) \ge k$ but the arboricity of $M$ is only $a(M) = \ceil{(2^s-1)/s}$. However, for $c_2(M)$, we show that the lower bound is tight.

\begin{theorem}\label{thm:oddcoverasymptotic}
    For every Eulerian binary matroid $M \subs \F_2^n \sm \set{0}$ it holds that
    \begin{equation*}
        c_2(M) \le \frac 43 a(M) \qquad \text{ and } \qquad c_2(M) = (1 + o(1)) a(M) \quad \text{ as } \quad a(M) \to \infty.
    \end{equation*}
\end{theorem}

The rest of the paper is organized as follows. In \cref{sec:decompmatroids} we construct circuit decompositions for arbitrary binary matroids and prove \cref{thm:decompqrdivr,thm:decompmlogrdivroverq}. We then specialise to the complete binary matroid and provide a proof of \cref{thm:decompcomplete} in \cref{sec:decompcompletematroids}. \cref{thm:oddcoverasymptotic} is proven in \cref{sec:oddcovers}, and we conclude in \cref{sec:openproblems} with some open problems.

\section{Decomposing arbitrary binary matroids into circuits}
\label{sec:decompmatroids}

To decompose any binary matroid $M$ into circuits, our main method is to greedily remove the largest circuit in $M$ that we can find. The following lemma gives an implicit lower bound on the size of such a circuit.

\begin{lemma}\label{lem:circuitsizeoverq}
    Let $M \subs \F_2^n \sm \set{0}$ be a binary matroid and $c \ge 2$ be an integer. If $M$ contains no circuit of size larger than $c$, then
    \begin{equation*}
        \abs{M} \le \sum_{i = 1}^{c-1} \binom{\rank(M)}{i}.
    \end{equation*}
\end{lemma}

\begin{proof}
    Let $r = \rank(M)$ and let $B = \set{b_1, \dots, b_r} \subs M$ be a basis of $M$. For every $m \in M \sm B$, there exists a unique nonempty subset $I \subs [r]$ such that $m = \sum_{i \in I} b_i$.
    
    We claim that $C = \set{m} \cup \set{b_i: i \in I}$ is a circuit. Indeed, $m + \sum_{i \in I} b_i = 0$. Moreover, if $\emptyset \neq D \subs C$ with $\sum_{x \in D} x = 0$, it cannot hold that $D \subs B$ since $B$ is an independent set. Hence, $D = \set{m} \cup \set{b_j: j \in J}$ for some set $J \subs I$. This implies that $m$ is in the span of $\{b_j : j \in J\}$. So, by uniqueness of $I$, we must have $J = I$ and thus $D = C$. This shows that $C$ is a circuit.

    Because $M$ contains no circuit of size larger than $c$, we know that $\abs{I} + 1 = \abs{C} \le c$ and thus $\abs{I} \le c - 1$. As $m \notin B \cup \{0\}$, we also know that $\abs{I} \ge 2$. Moreover, the set $I$ entirely determines $m$. Therefore,
    \begin{equation*}
        \abs{M} = \abs{B} + \abs{M \sm B} \le r + \sum_{i = 2}^{c-1} \binom{r}{i} = \sum_{i = 1}^{c-1} \binom{r}{i}. \qedhere
    \end{equation*}
\end{proof}

If we now apply the greedy algorithm that always removes the largest circuit of $M$, whose size we lower bound by the preceding lemma, we can prove \cref{thm:decompmlogrdivroverq}.

\begin{proof}[Proof of \cref{thm:decompmlogrdivroverq}]
     We assume that $M$ is nonempty. Let $r = \rank(M) \ge 2$. We claim that if $N \subs M$ is Eulerian and nonempty, then $N$ contains a circuit of size at least $\log \abs{N} / \log r$. If not, this value would have to be larger than three since $N$ contains some circuit and every circuit has size at least three. But then $N$ would contain no circuit of size larger than $c = \floor{\log \abs{N} / \log r} \ge 3$ and so \cref{lem:circuitsizeoverq} would imply that
    \begin{align*}
        \abs{N} &\le \sum_{i = 1}^{c-1} \binom{\rank(N)}{i} \le \sum_{i = 1}^{c-1} \binom{r}{i} \le \sum_{i = 1}^{c-1} r^i \\
        &= \frac{r^c - r}{r - 1} < r^c \le \abs{N},
    \end{align*}
    giving a contradiction.

    To decompose $M$ into circuits, we start with $N = M$ and repeatedly remove a maximum circuit from $N$ until $N$ is empty. During this process, $N$ remains Eulerian. While $N$ satisfies $\abs{N} \ge \abs{M} / \log^2 \abs{M}$, we know from the discussion above that $N$ contains a circuit of size at least
    \begin{equation*}
        \frac{\log \abs{N}}{\log r} \ge \frac{\log \abs{M} - 2 \log \log \abs{M}}{\log r}.
    \end{equation*}
    Hence, after at most
    \begin{equation*}
        \frac{\abs{M}}{\frac{\log \abs{M} - 2 \log \log \abs{M}}{\log r}} = (1 + o(1)) \frac{\abs{M} \log r}{\log \abs{M}}
    \end{equation*}
    many steps, $N$ will satisfy $\abs{N} \le \abs{M} / \log^2 \abs{M}$. Note that
    \begin{equation*}
        \abs{N} \le \frac{\abs{M}}{\log^2 \abs{M}} \le \frac{2 \abs{M} \log r}{\log^2 \abs{M}} = o(1) \frac{\abs{M} \log r}{\log \abs{M}}.
    \end{equation*}
    Hence, by decomposing $N$ into at most $|N|/3$ circuits, we decompose $M$ into at most
    \begin{equation*}
        (1 + o(1)) \frac{\abs{M} \log r}{\log \abs{M}}
    \end{equation*}
    many circuits, as required.
\end{proof}

Next, we want to prove \cref{thm:decompqrdivr}. If $M$ has size $\cO(2^{\rank(M)} / \log(\rank(M)))$, \cref{thm:decompmlogrdivroverq} already tells us that $c(M) = \cO(2^{\rank(M)} / (\rank(M) + 1))$. Thus, it suffices to prove this bound if $M$ is very dense, meaning that its size is close to $2^{\rank(M)}$.

In this setting, we still want to use the greedy algorithm to decompose $M$ into circuits. However, the lower bound on the circuit size used in the preceding proof will no longer be sufficient. Instead, if $M$ is dense, we need to show that there are circuits of size $\Theta(\rank(M))$ to obtain the desired result. To this end, we use the following standard entropy bound on the sum of binomial coefficients. Here, we denote the binary entropy function by $H(\alpha) = -\alpha \log_2 \alpha - (1-\alpha) \log_2 (1-\alpha)$.

\begin{lemma}\label{lem:generalentropybound}
    Let $r$ be a positive integer. Then for any $\alpha \in [0,1/2]$ we have
    \begin{equation*}
        \sum_{i = 0}^{\floor{\alpha r}} \binom ri \le 2^{H(\alpha)r}.
    \end{equation*}
\end{lemma}

\begin{proof}
    Note that $\alpha \le 1-\alpha$ and therefore
    \begin{align*}
        \sum_{i = 0}^{\floor{\alpha r}} \binom{r}{i} & \le \sum_{i = 0}^{\floor{\alpha r}} \binom{r}{i} \l(\frac{1-\alpha}{\alpha}\r)^{\alpha r - i} \\
        & = \frac{1}{\alpha^{\alpha r} (1-\alpha)^{(1-\alpha) r}} \sum_{i = 0}^{\floor{\alpha r}} \binom{r}{i} (1-\alpha)^{r-i} \alpha^i \\
        & \le \frac{1}{\alpha^{\alpha r} (1-\alpha)^{(1-\alpha) r}} = 2^{H(\alpha) r}. \qedhere
    \end{align*}
\end{proof}

By combining this entropy bound with \cref{lem:circuitsizeoverq}, we can now prove that every dense binary matroid $M$ has circuits of size $\Theta(\rank(M))$ and can therefore be decomposed into $\cO(\abs{M} / (\rank(M) + 1))$ many circuits.

\begin{theorem}\label{thm:decompMdivroverq}
    For any $\eps > 0$, there exist $r_0 \in \N$ and $\delta > 0$ such that every Eulerian binary matroid $M \subs \F_2^n \sm \set{0}$ with $\rank(M) \ge r_0$ and $\abs{M} \ge 2^{(1 -\delta)\rank(M)}$ satisfies
    \[
        c(M) \le \left(2 + \eps\right)\frac{\abs{M}}{\rank(M) + 1}.
    \]
\end{theorem}

\begin{proof}
    Let $\alpha = 1 / (2 + \eps/2)$ and $\delta = (1 - H(\alpha)) / 2$. It is easily verified that $H(x)$ is strictly increasing on $[0,1/2]$ with $H(1/2) = 1$, and so we have $\delta > 0$. Let $M$ be an Eulerian binary matroid with $r = \rank(M)$ and $\abs{M} \ge 2^{(1 -\delta)r}$. By \cref{lem:generalentropybound},
    \[
        \sum_{i =1}^{\lfloor\alpha r\rfloor} \binom{r}{i} \le 2^{H(\alpha)r} = 2^{(1-2\delta)r} < \abs{M},
    \]
    so we know by \cref{lem:circuitsizeoverq} that $M$ contains a circuit of size at least $\alpha r$. We remove circuits of this size until the remaining Eulerian binary matroid $N$ has size $\abs{N} \le 2^{(1 - 2\delta)r}$. The number of circuits removed so far is at most $\abs{M} / (\alpha r)$, and $N$ can be decomposed into at most $\abs{N} / 3$ circuits. Now,
    \[
        \frac{\abs{N}}{3} \le \frac{2^{(1 - 2\delta)r}}{3} \le \frac{2^{-\delta r}\abs{M}}{3} = o\l(\frac{\abs{M}}{r+1}\r).
    \]
    Thus we have
    \begin{align*}
        c(M) &\le \frac{\abs{M}}{\alpha r} + o\left(\frac{\abs{M}}{r+1}\right) \\
        &= \left(\frac{r+1}{r}(2 + \eps/2) + o(1)\right)\frac{\abs{M}}{r+1} \\
        &\le \left(2 + \eps\right)\frac{\abs{M}}{r+1}
    \end{align*}
    for $r$ sufficiently large.
\end{proof}

In particular, if $M$ is dense, this result implies that $c(M)$ is within a factor of $2 + o(1)$ of the lower bound $\abs{M} / (\rank(M) + 1)$ from the introduction. \Cref{thm:decompqrdivr} is now an easy consequence.

\begin{proof}[Proof of \cref{thm:decompqrdivr}]
    Let $\delta$ and $r_0$ be as in \cref{thm:decompMdivroverq} with $\varepsilon = 1/2$. If $\rank(M) \ge r_0$ and $\abs{M} \ge 2^{(1-\delta)\rank(M)}$, the theorem implies $c(M) \le \cO(2^{\rank(M)} / (\rank(M) + 1))$. Otherwise, $M$ can be decomposed into at most $\abs{M} / 3$ circuits, and
    \begin{equation*}
        \frac{\abs{M}}{3} \le 2^{(1-\delta)\rank(M)} = \frac{2^{\rank(M)}}{2^{\delta \rank(M)}} = o\l(\frac{2^{\rank(M)}}{\rank(M) + 1}\r). \qedhere
    \end{equation*}
\end{proof}

\section{Decomposing complete binary matroids into circuits}
\label{sec:decompcompletematroids}

In this section we prove \cref{thm:decompcomplete}, so we decompose the complete binary matroid $M$ into circuits. We will construct the circuits of this decomposition as orbits under a particular group action on $M$. This special structure allows us to show that $M$ can be decomposed into exactly $\abs{M} / (\rank(M) + 1)$ many circuits, as required.

\begin{proof}[Proof of \cref{thm:decompcomplete}]
    For $i \in \Z_p$, we write $e_i \in \F_2^p$ for the $i$-th standard basis vector of $\F_2^p$. For $x \in \F_2^p$, we denote by $x_i = \inprod{x, e_i}$ the $i$-th coordinate of $x$. Define
    \begin{equation*}
        N = \set*{x \in \F_2^p \sm \set{0} \st \sum_{i \in \Z_p} x_i = 0}.
    \end{equation*}
    Note that $N$ is isomorphic to $M$ since $\rank(N) = p-1$ and $|N| = 2^{\rank(N)}-1$.

    Consider the following group action of $\Z_p$ on $N$ defined for $j \in \Z_p$ by the linear map
    \begin{equation*}
        \phi_j(x) = \sum_{i \in \Z_p} x_i e_{i+j} = (x_{p-j}, x_{p-j+1}, \dots, x_{p-1}, x_0, x_1, \dots, x_{p-j-1}).
    \end{equation*}
    We claim that the orbits in $N$ under this action are circuits of size $p$, which gives us a decomposition of $N$ into such circuits.

    Consider an orbit $O = \set{\phi_j(x) \st j \in \Z_p}$ for some $x \in N$. By definition of $N$, $\sum_{i \in \Z_p} x_i = 0$, and so because $p$ is odd we must have $x_i \neq x_{i+1}$ for some $i$. This implies that $\phi_1(x) \neq x$ and therefore $\abs{O} \ge 2$. Since $\abs{O}$ divides $p$ by the Orbit-Stabilizer Theorem and because $p$ is prime, it follows that $\abs{O} = p$.

    It remains to show that $O$ is a circuit. First, we show that since $2$ has multiplicative order $p-1$ in $\Z_p$, the polynomial $(t^p-1)/(t-1) = t^{p-1} + \dots + t + 1$ is irreducible over $\F_2$. Let $\omega \neq 1$ be a $p$-th root of unity over $\mathbb F_2$, and let $g(t)$ be its minimal polynomial over $\mathbb F_2$. Recall that $g(t)$ is irreducible over $\mathbb F_2$ and has the form $\prod_{i = 1}^k (t-\omega_i)$ over the splitting field $K$ of $\omega$, where $\omega_1, \ldots, \omega_k \in K$ are the distinct Galois conjugates of $\omega$ over $\mathbb F_2$. Repeated application of the Frobenius automorphism $\alpha \mapsto \alpha^2$ in $K$ gives $p-1$ distinct Galois conjugates of $\omega$, namely $\omega, \omega^2, \omega^{2^2}, \ldots, \omega^{2^{p-2}}$. Thus $\deg g(t) = k \ge p-1$, so $g(t) = (t^p-1)/(t-1)$ is irreducible over $\mathbb F_2$.
    
    Now, let $\phi = \sum_{i \in \Z_p} x_i \phi_i = \sum_{i \in \Z_p} x_i \phi_1^i$, where we identify $\mathbb Z_p$ with $\{0, \ldots, p-1\}$. Note that $O$ can be rewritten as
    \begin{equation*}
        O = \set*{\sum_{i \in \Z_p} x_i e_{i+j} \st j \in \Z_p} = \set*{\sum_{i \in \Z_p} x_i \phi_i(e_j) \st j \in \Z_p} = \set{\phi(e_j) \st j \in \Z_p}.
    \end{equation*}
    Therefore, a linear dependence in $O$ is of the form
    \begin{equation*}
        0 = \sum_{j \in \Z_p} \mu_j \phi(e_j) = \phi(y)
    \end{equation*}
    where $y = \sum_{j \in \Z_p} \mu_j e_j$. To show that $O$ is a circuit, it suffices to show that the kernel of $\phi$ is generated by $\sum_{i \in \Z_p} e_i$. Here, we follow the approach of \cite{metamorphy}. To this end, let $f(t) = \sum_{i \in \Z_p} x_i t^i$. Since $f(1) = \sum_{i \in \Z_p} x_i = 0$, we know that $t-1$ divides $f(t)$, and because $(t^p-1)/(t-1)$ is irreducible, we must have $\mathrm{gcd}(f(t), t^p-1) = t-1$. By properties of the gcd, there exist polynomials $h(t), u(t), v(t) \in \F_2[t]$ such that 
    \begin{align*}
        f(t) &= h(t)(t-1); \\
        t-1 &= u(t)f(t) +  v(t)(t^p-1).
    \end{align*}
    Note that $\phi_1^p = \phi_0$, which is the identity on $\F_2^p$, so we have 
    \begin{align*}
        f(\phi_1) &= h(\phi_1)(\phi_1 - \phi_0); \\
        \phi_1 - \phi_0 &= u(\phi_1)f(\phi_1).
    \end{align*}
    Thus $\phi = f(\phi_1)$ and $\phi_1-\phi_0$ have the same kernel, which is easily seen to be spanned by $\sum_{i \in \Z_p} e_i$. This proves that $O$ is a circuit, so we've successfully decomposed $N$ into circuits of size $p$.
\end{proof}

We note that the technical conditions of \cref{thm:decompcomplete} are necessary for our proof. For example, our method fails for $p=7$, where the orbit of $e_0 + e_1 + e_2 + e_4$ decomposes into two circuits. Perhaps a different construction could give the same bound, up to rounding, for arbitrary complete binary matroids.

\section{Odd-covers of binary matroids}
\label{sec:oddcovers}

In this section, we consider circuit odd-covers and prove \cref{thm:oddcoverasymptotic}. Again, as for circuit decompositions, our strategy will be to greedily find a large circuit $C$, but instead of removing $C$ from $M$, we instead replace $M$ by $M \oplus C$. This means that $C$ can also use elements outside of $M$ and it is only important that $M \oplus C$ is much smaller than $M$ so that the greedy algorithm finishes quickly. To find a suitable circuit $C$, we simply pick a maximal independent set of $M$ and complete it to a circuit. This leads to the following bound.

\begin{lemma}\label{lem:symdifpreliminarybound}
    For every Eulerian binary matroid $M \subs \F_2^n \sm \set{0}$ it holds that
    \[
        c_2(M) \le (1 + o(1))\frac{\abs{M}}{\log_2\abs{M}}.
    \]
\end{lemma}

\begin{proof}
    We start with $N = M$ and repeatedly replace $N$ with $N \oplus C$ for some circuit $C \subs \F_2^n \setminus \{0\}$. This retains that $N$ is Eulerian. We obtain $C$ by taking a maximal linearly independent subset $I$ of $N$ and completing it to a circuit $C := I \cup \{x\}$ where $x = \sum_{y \in I} y$. Since $\rank(N) \ge \log_2\abs{N}$, we reduce the size of $N$ by at least $\log_2\abs{N} - 1$ at every step. As long as $\abs{N} \ge \abs{M} / \log^2 \abs{M}$, we have 
    \[
        \log_2 \abs{N} - 1 \ge \log_2 \abs{M} - 2 \log_2 \log \abs{M} - 1.
    \]
    Once $\abs{N} < \abs{M} / \log^2 \abs{M}$, we can decompose $N$ into at most $\abs{N} / 3$ circuits of size at least $3$. In total, the number of circuits used is at most
    \[
        \frac{\abs{M}}{\log_2\abs{M} - 2\log_2 \log \abs{M} - 1} + \frac{\abs{M}}{3\log^2\abs{M}} = (1 + o(1))\frac{\abs{M}}{\log_2 \abs{M}}. \qedhere
    \]
\end{proof}

We now establish the exact asymptotics of $c_2(M)$ in the regime where $a(M) \to \infty$.

\begin{proof}[Proof of \cref{thm:oddcoverasymptotic}]
    We have $c_2(M) \ge (1 + o(1))a(M)$ already from \eqref{eq:arblowerboundoddcover}. To prove the upper bounds, let $M = I_1 \cup \cdots \cup I_t$ be a decomposition of $M$ into $t = a(M)$ linearly independent sets. For each $I_i$, the set $C_i := I_i \cup \{x_i\}$ where $x_i = \sum_{y \in I_i} y$ is a circuit. Now, the matroid $N := M \oplus C_1 \oplus \cdots \oplus C_t \subs \{x_1, \ldots, x_t\}$ is Eulerian of size at most $t$, so $N$ can be decomposed into at most $t / 3$ circuits, implying that $c_2(M) \le (4/3) t$. But actually, by \cref{lem:symdifpreliminarybound}, $N$ is the symmetric difference of at most $(1 + o(1)) (t / \log_2 t)$ circuits, giving
    \[
        c_2(M) \le t + (1 + o(1))\frac{t}{\log_2 t} = (1 + o(1))a(M). \qedhere
    \]
\end{proof}

\section{Open problems}
\label{sec:openproblems}

We showed that for certain values of $n$, the complete $n$-dimensional binary matroid $M$ can be decomposed into exactly $(2^{\rank(M)} - 1) / ((\rank(M)+1))$ many circuits. We conjecture that this is an upper bound for the minimum size of a circuit decomposition of any Eulerian binary matroid.

\begin{conjecture}
    For every Eulerian binary matroid $M \subs \F_2^n \sm \set{0}$ it holds that
    \[
        c(M) \le \ceil*{\frac{2^{\rank(M)}-1}{\rank(M)+1}}.
    \]
\end{conjecture}

In particular, we believe that the conclusion of \Cref{thm:decompcomplete} should hold, up to rounding, for any complete binary matroid, without the technical assumptions on $p$.

For odd-covers of binary matroids, we determined that $c_2(M) = (1 + o(1)) a(M)$ as $a(M) \to \infty$. There are numerous matroids with $c_2(M) \ge a(M)$. For example, if $M$ consists of two independent copies of $\F_2^s \sm \set{0}$, it is easy to see that any circuit covers at most $\rank(M)$ many elements of $M$ and so $c_2(M) \ge a(M)$. The difference between this and the lower bound from \cref{prop:densitylowerboundoverq} grows arbitrarily large as $s \to \infty$. However, we believe that there are no matroids which are worse than this, meaning that there should always be an odd-cover of size at most $a(M)$.

\begin{conjecture}
    For every Eulerian binary matroid $M \subs \F_2^n \sm \set{0}$ it holds that
    \[
        c_2(M) \le a(M).
    \]
\end{conjecture}

\textbf{Acknowledgements.} We would like to thank Sang-il Oum, Alex Scott, David Wood, and Liana Yepremyan for organising the April 2023 \href{https://www.matrix-inst.org.au/events/structural-graph-theory-downunder-iii/}{MATRIX-IBS Structural Graph Theory Downunder III} workshop where part of this research was performed.

{
\fontsize{11pt}{12pt}
\selectfont

\hypersetup{linkcolor=structurecolor}
\newcommand{\etalchar}[1]{$^{#1}$}

}
\end{document}